\documentclass{amsart}

\usepackage{amsmath}
\usepackage{amscd}
\usepackage{graphicx}
\newtheorem{theorem}{Theorem}[section]
\newtheorem{lemma}[theorem]{Lemma}

\theoremstyle{definition}
\newtheorem{definition}[theorem]{Definition}

\theoremstyle{proposition}
\newtheorem{proposition}[theorem]{Proposition}

\theoremstyle{corollary}

\numberwithin{subsection}{section}
\numberwithin{equation}{section}

\begin{document}

\title{Degenerate abelian function fields}

\author{Yukitaka \textsc{Abe}}
\address{Graduate School of Science and Engineering for Research, 
University of Toyama, Toyama 930-8555, Japan}
\email{abe@sci.u-toyama.ac.jp}


\subjclass[2010]{Primary 
32A20; Secondary 30D30, 14H40}

\keywords{degenerate abelian functions, algebraic addition theorem, singular curves, generalized Jacobi varieties, Albanese varieties,
}


\begin{abstract}
Originally, an abelian function field is the field of meromorphic functions
on the Jacobi variety $J(X)$ of a compact Riemann surface $X$.
It is generated by the fundamental abelian functions belonging to the
meromorphic function field on $X$.
We study this relation for singular curves.
\end{abstract}

\maketitle

\section{Introduction}
Abelian functions were discovered by solving Jacobi's inversion problem.
They are meromorphic functions on the Jacobi variety $J(X)$ of a
compact Riemann surface $X$. 
We denote by
${\rm Mer}(J(X))$ the field of meromorphic functions
on $J(X)$.
Let ${\rm Mer}(X)$ be the field of meromorphic functions on $X$.
The field ${\rm Mer}(J(X))$ is generated by the
fundamental abelian functions belonging to ${\rm Mer}(X)$.
We study this relation for singular curves.\par
In general, abelian function fields are fields of meromorphic functions
on abelian varieties.
We also know that an abelian function field admits an algebraic addition
theorem. Weierstrass stated the following in his lectures in Berlin
(see \cite{ref8}):\\
Every system of $n$ (independent) functions in $n$ variables which
admits an algebraic addition theorem is an algebraic combination of
$n$ abelian (or degenerate abelian) functions with the same periods.\\
However, he did not publish his proof of the above statement. We did
not know the precise meaning of degenerate abelian functions at that time.\par
We determined meromorphic function fields which admit an algebraic
addition theorem in \cite{ref2} and \cite{ref3}. However, our statement
was in a little weak form. We give its final form (Theorem 2.5)  in this paper.\par
Another purpose of this paper is to complete the last section in the
previous paper \cite{ref6}. 
Let $S$ be a finite subset of a compact Riemann surface $X$.
We construct a singular curve
$X_{{\mathfrak m}}$ from $X$ by a modulus
${\mathfrak m}$ with support $S$.
Let $A = {\rm Alb}^{an}(X_{{\mathfrak m}})$ be the analytic Albanese
variety of $X_{{\mathfrak m}}$. We proved in \cite{ref6} that the
period map $\varphi : X \setminus S \longrightarrow A$ gives a
bimeromorphic map $\varphi : (X \setminus S)^{(\pi )} \longrightarrow A$,
where $\pi $ is the genus of $X_{{\mathfrak m}}$ and $(X \setminus S)^{(\pi )}$
is the symmetric product of $X \setminus S$ of degree $\pi $.
The analytic Albanese variety $A$ has the standard compactification
$\overline{A}$ which is projective algebraic. Let ${\rm Mer}(\overline{A})$
be the field of meromorphic functions on $\overline{A}$. Its restriction
${\rm Mer}(\overline{A})|_{A}$ onto $A$ is a function field which admits
an algebraic addition theorem. We show that
${\rm Mer}\left((X_{{\mathfrak m}})^{(\pi )}\right) \cong {\rm Mer}(\overline{A})$
through $\varphi $, where $(X_{{\mathfrak m}})^{(\pi )}$ is the symmetric
product of $X_{{\mathfrak m}}$ of degree $\pi $ (Theorem 5.2).\par
We use notations in the previous paper \cite{ref6}.

\section{Algebraic addition theorem}
Let $K$ be a subfield of the meromorphic function field
${\rm Mer}({\mathbb C}^n)$ on ${\mathbb C}^n$. We consider the
following condition (T) concerning the transcendence degree
${\rm Trans}_{{\mathbb C}}K$ of $K$ over ${\mathbb C}$.\par
\medskip
\noindent
(T) $K$ is finitely generated over ${\mathbb C}$ and 
${\rm Trans}_{{\mathbb C}}K = n$.\par
\medskip
\noindent
If $K$ satisfies condition (T), then we can take functions
$f_0, f_1, \dots , f_n \in K$ with $K = {\mathbb C}
(f_0, f_1, \dots , f_n)$.

\begin{definition}
Let $K = {\mathbb C}(f_0, f_1, \dots , f_n)$ be a subfield
of ${\rm Mer}({\mathbb C}^n)$ satisfying condition
{\rm (T)}. We say that $K$ admits an algebraic addition
theorem (this is abbreviated to {\rm (AAT)}) if for any
$j = 0, 1, \dots , n$ there exists a rational function $R_j$
such that
\begin{equation}
f_j(z + w) = R_j(f_0(z), f_1(z), \dots , f_n(z), f_0(w),
f_1(w), \dots ,f_n(w))
\end{equation}
for all $z,w \in {\mathbb C}^n$.
\end{definition}

The above definition does not depend on the choice of
generators $f_0, f_1, \dots , f_n$ of $K$.\par
Any connected commutative complex Lie group $G$ of dimension $n$ is represented as
\begin{equation*}
 G = {\mathbb C}^n/\Gamma = {\mathbb C}^p \times
({\mathbb C}^{*})^q \times ({\mathbb C}^r / \Gamma _0),
\end{equation*}
where ${\mathbb C}^r/\Gamma _0$ is a toroidal group and
$p+q+r = n$. By the standard compactification
$\overline{{\mathbb C}^r/\Gamma _0}$ of ${\mathbb C}^r/\Gamma _0$
we obtain a compactification
\begin{equation*}
\overline{G} = ({\mathbb P}^1)^{p} \times ({\mathbb P}^1)^{q}
\times \overline{{\mathbb C}^r/\Gamma _0}
\end{equation*}
of $G$, which is called the standard compactification of
$G$ (for details, see \cite{ref4}).\par
Let $f \in {\rm Mer}({\mathbb C}^n)$. We define the period group $\Gamma _f$ of $f$ by
$$\Gamma _f := \{ \gamma \in {\mathbb C}^n;\, f(z + \gamma ) =
f(z)\quad \text{for all $z \in {\mathbb C}^n$}\}.$$

\begin{definition}
A meromorphic function $f$ on ${\mathbb C}^n$ is said to be
non-degenerate if its period group $\Gamma _f$ is discrete.
\end{definition}

For a subfield $K$ of ${\rm Mer}({\mathbb C}^n)$ we denote by
$\Gamma _{K} := \bigcap _{f \in K}\Gamma _f$ the period group of $K$.
A subfield $K$ is said to be non-degenerate if it has a non-degenerate
meromorphic function. We stated the following lemma without proof in
\cite{ref5}. Here we give its proof for the convenience of readers.

\begin{lemma}
A subfield $K$ is non-degenerate if and only if $\Gamma _{K}$ is discrete.
\end{lemma}

\begin{proof}
It is obvious by the definition that if $K$ is non-degenerate, then $\Gamma _{K}$
is discrete.\par
Conversely, suppose that $\Gamma _{K}$ is discrete.
By induction on $n$ we show that $K$ has a non-degenerate function.\par
When $n = 1$, the statement is trivial by the uniqueness theorem.\par
Let $n > 1$. We assume that it holds for subfields of ${\rm Mer}({\mathbb C}^k)$
with $k < n$. Take a non-constant $f \in K$. If $\Gamma _f$ is discrete, then
$f$ is a non-degenerate function. 
If $\Gamma _f$ is not discrete, then there exist complex linear subspaces
$V_1$ and $W_1$ such that ${\mathbb C}^n = V_1 \oplus W_1$ with
$0 < \dim V_1 < n$ and $\Gamma _f = V_1 \oplus \Lambda _1$, where
$\Lambda _1$ is a discrete subgroup of $W_1$. Since $\Gamma _{K|_{V_1}}$ is
discrete, there exists $g \in K$ such that $\Gamma _{g|_{V_1}}$ is discrete
by the assumption of induction.\par
Assume that $\Gamma _g$ is not discrete. Then there exist complex linear subspaces
$V_2$ and $W_2$ such that ${\mathbb C}^n = V_2 \oplus W_2$ and
$\Gamma _g = V_2 \oplus \Lambda _2$, where $\Lambda _2$ is a discrete
subgroup of $W_2$. We have $W_2 \supset V_1$ by the choice of $g$.
We set $h_1 := fg$. Then $h_1$ is non-degenerate on $V_1 \oplus V_2$,
because $f$ is constant on $V_1$ and $g$ is constant on $V_2$.
If ${\mathbb C}^n = V_1 \oplus V_2$, then $h_1$ is the desired function.
Otherwise, there exists a positive dimensional complex linear subspace $W$ such that
${\mathbb C}^n = V_1 \oplus V_2 \oplus W$.\par
If $\Gamma _{h_1}$ is not discrete, then there exist a complex linear
subspace $U_1$ and a discrete subgroup $\Gamma _1$ such that
$\Gamma _{h_1} = U_1 \oplus \Gamma _1$. Since $h_1$ is non-degenerate
on $V_1 \oplus V_2$, we have $U_1 \subset W$. Let
$h_2 := h_1g = fg^2$. Then $h_2$ is non-degenerate on $V_1 \oplus V_2$
by the same reason as above. On $U_1$ we have 
$h_2|_{U_1} = (h_1|_{U_1}) (g|_{U_1})$. Since $h_1|_{U_1}$ is constant and
$g|_{U_1}$ is non-degenerate, $\Gamma _{h_2|_{U_1}}$ is discrete.
Hence, $h_2$ is non-degenerate on $V_1 \oplus V_2 \oplus U_1$.
If $\Gamma _{h_2}$ is not discrete, we consider $h_3 := h_2 g = fg^3$.
Repeating this procedure, we finally obtain $m \in {\mathbb N}$ such
that $fg^m$ is non-degenerate.
\end{proof}

Let $G = {\mathbb C}^n/\Gamma $ be a connected commutative complex
Lie group as above. We denote by ${\rm Mer}(G)$ the field of meromorphic
functions on $G$. Let $\sigma : {\mathbb C}^n \longrightarrow G$ be
the projection. Then, for any subfield $K$ of ${\rm Mer}({\mathbb C}^n)$
with $\Gamma \subset \Gamma _{K}$ there exists a subfield $\kappa $
of ${\rm Mer}(G)$ such that $K = \sigma ^{*}\kappa $.
Let ${\rm Mer}(\overline{G})$ be the field of meromorphic functions on
the standard compactification $\overline{G}$ of $G$. We denote by
${\rm Mer}(\overline{G})|_{G}$ the restriction of ${\rm Mer}(\overline{G})$
onto $G$.

\begin{definition}[\cite{ref5}]
A subfield $K$ of ${\rm Mer}({\mathbb C}^n)$ is said to be
a {\rm W}-type subfield if $K = \sigma ^{*}({\rm Mer}(\overline{G})|_G)$,
where $G = {\mathbb C}^p \times ({\mathbb C}^{*})^q \times
{\mathcal Q}$ with an $r$-dimensional quasi-abelian variety
${\mathcal Q}$ of kind 0, $n = p + q + r$ and $\sigma : {\mathbb C}^n
\longrightarrow G$ is the projection
(for the definition of quasi-abelian varieties of kind 0, see
\cite{ref4}).
\end{definition}

The following theorem is the final form of the Weierstrass
statement (cf. \cite{ref2} and \cite{ref3}). We have already obtained enough material
to prove it in the previous papers (\cite{ref1}, \cite{ref2} and \cite{ref3}).

\begin{theorem}[cf. Theorem 1.1 in \cite{ref3}]
Let $K$ be a non-degenerate subfield of ${\rm Mer}({\mathbb C}^n)$
satisfying condition {\rm (T)}. If $K$ admits {\rm (AAT)},
then there exists a ${\mathbb C}$-linear isomorphism
$\Phi : {\mathbb C}^n \longrightarrow {\mathbb C}^n$ such that
$\Phi ^{*}K$ is a {\rm W}-type subfield.
\end{theorem}

\begin{proof}
We showed the following results in \cite{ref1}.
There exists a connected commutative complex Lie group
$\Omega = {\mathbb C}^n/\Gamma $ embedded in a complex
projective space such that $K \cong {\mathbb C}(Y)$, where $Y$
is the Zariski closure of $\Omega $. Furthermore we can represent $\Omega $
as $\Omega = {\mathbb C}^p \times ({\mathbb C}^{*})^q \times {\mathcal Q}$,
where ${\mathcal Q} = {\mathbb C}^r/\Gamma _0$ is a quasi-abelian variety.
Since ${\mathcal Q}$ is a connected closed complex Lie subgroup,
the Zariski closure $Z$ of ${\mathcal Q}$ has the same dimension as
${\mathcal Q}$ by Theorem 4.5 in \cite{ref2}. Therefore ${\mathcal Q}$
is a quasi-abelian variety of kind 0 (Theorem 8.1 in \cite{ref3}). Let
$\overline{{\mathcal Q}}$ be the standard compactification of
${\mathcal Q}$. Then we obtain
$Y = ({\mathbb P}^1)^p \times ({\mathbb P}^1)^q \times \overline{{\mathcal Q}}$.
Hence $K$ is a W-type subfield.
\end{proof}

%

\section{Singular curves}
Let $X$ be 
 a compact Riemann surface with the structure sheaf ${\mathcal O}_X$.
Take a finite subset
$S$ of $X$. We consider an equivalence
relation $R$ on $S$. We define the quotient set
$\overline{S} := S/R$ of $S$ by $R$. We set
$$\overline{X} := (X \setminus S)\cup \overline{S}.$$
We induce to $\overline{X}$ the quotient topology
by the canonical projection $\rho : X \longrightarrow \overline{X}$.
Then $\overline{X}$ is a compact Hausdorff space.

\begin{definition}[\cite{ref9}]
A modulus ${\mathfrak m}$ with support $S$ is the data of an
integer ${\mathfrak m}(P)>0$ for each point $P \in S$.
\end{definition}

Let $\rho _* {\mathcal O}_X$ be the direct image of
${\mathcal O}_X$ by the projection
$\rho : X \longrightarrow \overline{X}$.
For any $Q \in \overline{S}$ we denote by ${\mathcal I}_Q$
the ideal of $(\rho _* {\mathcal O}_X)_Q$ formed by 
functions $f$ with ${\rm ord}_P(f) \geq {\mathfrak m}(P)$ for
any $P \in \rho ^{-1}(Q)$.
We define a sheaf ${\mathcal O}_{\mathfrak m}$ on $\overline{X}$ by
$$
{\mathcal O}_{{\mathfrak m} , Q} :=
\begin{cases}
(\rho _*{\mathcal O}_X)_Q = {\mathcal O}_{X, Q}&
\mbox{if $Q \in X \setminus S$},\\
{\mathbb C} + {\mathcal I}_Q&
\mbox{if $Q \in \overline{S}$}.
\end{cases}$$
Then we obtain a 1-dimensional compact reduced complex
space $(\overline{X}, {\mathcal O}_{\mathfrak m})$, which
we denote by $X_{{\mathfrak m}}$.
\par
Conversely, any reduced and irreducible singular curve is obtained
as above.\par
Let $g$ be the genus of $X$. For any $Q \in X_{{\mathfrak m}}$ we set
$\delta _{Q} := \dim \left((\rho _{*}{\mathcal O}_X)_{Q}/{\mathcal O}_{{\mathfrak m}, Q}\right)$.
The genus of $X_{{\mathfrak m}}$ is defined by $\pi := g + \delta$,
where
$\delta := \sum_{Q \in X_{{\mathfrak m} }} \delta _Q $.\par
We denote by $\Omega _{{\mathfrak m}}$ the duality sheaf on $X_{{\mathfrak m}}$
(see Section 3.1 in \cite{ref6}). We have
$\dim H^0(X_{{\mathfrak m}}, \Omega _{{\mathfrak m}}) = \dim
H^1(X_{{\mathfrak m}}, {\mathcal O}_{{\mathfrak m}}) = \pi $ (cf. \cite{ref6}).

\section{Analytic Albanese varieties}

Let $X_{\mathfrak m}$ be a singular curve of genus $\pi = g + \delta $ constructed from $X$
by a modulus ${\mathfrak m}$ with support $S$, 
where $g$ is the genus of $X$. Take a basis
$\{ \omega _1, \dots , \omega _{\pi}\}$ of $H^0(X_{\mathfrak m},\Omega _{\mathfrak m})$.
We fix a canonical homology basis $\{ \alpha _1, \beta _1, \dots , \alpha _g,
\beta _g \}$ of $X$. Let $S = \{ P_1, \dots , P_s\}$.
We denote by $\gamma _j$ a small circle centered at $P_j$ with anticlockwise
direction for $j = 1, \dots ,s$. Then the set $\{ \alpha _1, \beta _1, \dots ,
\alpha _g, \beta _g, \gamma _1, \dots , \gamma _{s-1}\}$ forms a basis of
$H_1(X \setminus S, {\mathbb Z}) = H_1(X_{\mathfrak m} \setminus \overline{S},
{\mathbb Z})$. Let $H^0(X_{\mathfrak m},\Omega _{\mathfrak m})^{*}$ be the dual space of
$H^0(X_{\mathfrak m},\Omega _{\mathfrak m})$. We set
$$A := H^0(X_{\mathfrak m},\Omega _{\mathfrak m})^{*}/
H_1(X_{\mathfrak m} \setminus \overline{S},{\mathbb Z}).$$
Consider $2g + s -1$ vectors
$$\left(\int _{\alpha _i}\rho ^{*}\omega _1, \dots
, \int _{\alpha _i}\rho ^{*}\omega _{\pi }\right),\quad i = 1, \dots , g,$$
$$\left(\int _{\beta _i}\rho ^{*}\omega _1, \dots
, \int _{\beta _i}\rho ^{*}\omega _{\pi }\right),\quad i = 1, \dots , g$$
and
$$\left(\int _{\gamma _j}\rho ^{*}\omega _1, \dots
, \int _{\gamma _j}\rho ^{*}\omega _{\pi }\right),\quad j = 1, \dots , s-1$$
in ${\mathbb C}^{\pi}$. Let $\Gamma $ be a subgroup of ${\mathbb C}^{\pi}$
generated by these vectors over ${\mathbb Z}$.
Then $\Gamma $ is a discrete subgroup of ${\mathbb C}^{\pi }$.
We have $A \cong {\mathbb C}^{\pi }/\Gamma $ as a complex Lie group.
We call $A$ with the structure as a complex Lie group the analytic Albanese
variety of $X_{{\mathfrak m}}$, and write it as
${\rm Alb}^{an}(X_{{\mathfrak m}})$.
The theorem of Remmert-Morimoto says that
$$ A \cong
{\mathbb C}^{\pi }/\Gamma \cong {\mathbb C}^p \times
({\mathbb C}^{*})^q \times G,$$
where $G$ is a toroidal group of dimension $r$ and $p + q + r = \pi $.
In \cite{ref6} we showed that $G$ is a quasi-abelian variety of kind 0.
Then we write
$$ A =
{\rm Alb}^{an}(X_{{\mathfrak m}})  =  {\mathbb C}^p \times
({\mathbb C}^{*})^q \times {\mathcal Q},$$
where ${\mathcal Q}$ is an $r$-dimensional quasi-abelian variety
of kind 0.
We define a period map $\varphi $ with base point
$P_0 \in X \setminus S$ by 
$$\varphi : X \setminus S \longrightarrow
A,\quad
P \longmapsto \left[ \left( \int _{P_0}^{P}\rho ^{*} \omega _1, \dots ,
\int _{P_0}^{P}\rho ^{*} \omega _{\pi}\right)\right].$$
The period map $\varphi $ is extended to a bimeromorphic map
$\varphi : (X \setminus S)^{(\pi )} \longrightarrow A$, where
$(X \setminus S)^{(\pi )}$ is the symmetric product of
$X \setminus S$ of degree $\pi $ (Theorem 5.19 in \cite{ref6}).

\section{Degenerate abelian function fields}
We set $K_{{\mathfrak m}} := \rho ^{*}{\rm Mer}(X_{{\mathfrak m} }) \subset {\rm Mer}(X).$
Then there exist $x,y \in K_{{\mathfrak m}}$ such that $K_{{\mathfrak m}} = {\mathbb C}(x,y).$
Since $x$ and $y$ are algebraically dependent, we have an irreducible
polynomial $f$ such that $f(x,y) = 0.$ Let $C$ be the closure of
$\{ f = 0 \}$ in ${\mathbb P}^1 \times {\mathbb P}^1.$
The $\pi $-dimensional complex projective space ${\mathbb P}^{\pi }$ is
identified with the symmetric product $({\mathbb P}^1)^{(\pi)}$ of
${\mathbb P}^1$ of degree $\pi $ by the map induced from the following
rational map
$$\tau : ({\mathbb P}^1)^{\pi} \ni (x_1, \dots , x_{\pi}) \longmapsto
(a_0 : a_1 :\cdots :a_{\pi}) \in {\mathbb P}^{\pi},$$
$$\frac{a_1}{a_0} = - \sum _{i=1}^{\pi}x_i,\ 
\frac{a_2}{a_0} = \sum _{i<j}x_ix_j,\ \dots,\ 
\frac{a_{\pi}}{a_0} = (-1)^n \prod_{i=1}^{\pi}x_i,$$
where $(x_1, \dots , x_{\pi})$ is the inhomogeneous coordinates of
$({\mathbb P}^1)^{\pi}$ and $(a_0: a_1: \cdots : a_{\pi})$ is the homogeneous
coordinates of ${\mathbb P}^{\pi}$.
By the definition of $C$ we have a holomorphic map
$\mu : X \longrightarrow C$. This gives a holomorphic map
$$\mu _0 : X^{\pi } \longrightarrow C^{\pi } \subset ({\mathbb P}^1 \times {\mathbb P}^1)^{\pi } \cong
({\mathbb P}^1)^{\pi } \times ({\mathbb P}^1)^{\pi }.$$
Let $\sigma _0 : X^{\pi } \longrightarrow X^{(\pi )}$ be the canonical projection.
Then we obtain the following
commutative diagram:
\begin{equation*}
\begin{CD}
X^{\pi } @>\mu >>  C^{\pi } \subset ({\mathbb P}^1 \times {\mathbb P}^1)^{\pi }
\cong ({\mathbb P}^1)^{\pi } \times ({\mathbb P}^1)^{\pi }\\
@V{\sigma _0}VV              @VV{\tau \times \tau }V\\
X^{(\pi )}   @>\hat{\mu}>>     {\mathbb P}^{\pi } \times {\mathbb P}^{\pi } \\
\cup  @.
  @V{\sigma _1}V{\sigma _2}V \\
(X \setminus S)^{(\pi )} @.  {\mathbb P}^{\pi }\quad {\mathbb P}^{\pi }\\
@V{\varphi}VV @. @. \\
A = {\rm Alb}^{an}(X_{{\mathfrak m} }) @.  @.   \\
@A{\sigma } AA @.   @.   \\
{\mathbb C}^{\pi } @.   @.  \\
\end{CD}
\end{equation*}
where $\hat{\mu } : X^{(\pi )} \longrightarrow {\mathbb P}^{\pi } \times
{\mathbb P}^{\pi }$ is the holomorphic map induced from $\mu $ and
$\sigma _i :  {\mathbb P}^{\pi } \times {\mathbb P}^{\pi } \longrightarrow
{\mathbb P}^{\pi }$ is the projection onto the $i$-th component for
$i = 1, 2$. Therefore we can define 
meromorphic maps
$\wp ^x : {\mathbb C}^{\pi } \longrightarrow {\mathbb P}^{\pi }$ and
$\wp ^y : {\mathbb C}^{\pi } \longrightarrow {\mathbb P}^{\pi }$ by
$$\wp ^x := \sigma _1 \circ \hat{\mu}\circ \varphi ^{-1} \circ \sigma \quad
\text{and}\quad
\wp ^y := \sigma _2 \circ \hat{\mu}\circ \varphi ^{-1} \circ \sigma .$$
If we represent $\wp ^x$ and $\wp ^y$ in homogeneous coordinates of
${\mathbb P}^{\pi }$ as
$$\wp ^x(z) = (1: \xi _1(z): \cdots : \xi _{\pi}(z)) \quad
\text{and}\quad
\wp ^y(z) = (1: \eta _1(z): \cdots : \eta _{\pi }(z)),$$
then $\xi _1(z), \dots , \xi _{\pi }(z), \eta _1(z), \dots , \eta _{\pi }(z) \in \sigma ^{*}
{\rm Mer}(A).$
Let $[z]$ be a generic point of $A$, where $z \in {\mathbb C}^{\pi }$.
Then there exists uniquely $(P_1, \dots , P_{\pi }) \in
(X \setminus S)^{(\pi )}$ with $\varphi ((P_1, \dots , P_{\pi })) = [z]$.
In this case we have
$$\xi _1(z) = - \sum _{i=1}^{\pi }x(P_i),\ 
\xi _2(z) = \sum _{i<j}x(P_i)x(P_j), \dots ,\ 
\xi _{\pi }(z) = (-1)^{\pi }\prod _{i=1}^{\pi }x(P_i),$$
$$\eta _1(z) = - \sum _{i=1}^{\pi } y(P_i),\ 
\eta _2 (z) = \sum _{i<j}y(P_i)y(P_j), \dots ,\  
\eta _{\pi }(z) = (-1)^{\pi }\prod _{i=1}^{\pi }y(P_i).$$
We denote
$K := {\mathbb C}(\xi _1(z), \dots , \xi _{\pi }(z),
\eta _1(z), \dots , \eta _{\pi }(z))\subset \sigma ^{*}{\rm Mer}(A)$. Let $\kappa $ be the corresponding subfield of
${\rm Mer}(A)$, i.e. $K = \sigma ^{*}\kappa$.
Then we have $\varphi ^{*}\kappa = {\rm Mer}((X_{{\mathfrak m}})^{(\pi )})$
 by the above relation, where
$(X_{{\mathfrak m}})^{(\pi )}$ is the symmetric product of $X_{{\mathfrak m}}$
of degree $\pi $. We note that $(X_{{\mathfrak m}})^{(\pi )}$ is an irreducible
projective algebraic variety (cf. III. 14 in \cite{ref9}).
We may consider ${\rm Mer}((X_{{\mathfrak m}})^{(\pi )})$ as
a subfield of ${\rm Mer}(X^{(\pi )})$.
It is obvious that $\varphi ^{*}\kappa $ is finitely generated over
${\mathbb C}$ and ${\rm Trans}_{{\mathbb C}}\varphi ^{*}\kappa = \pi $.

\begin{proposition}
$K$ admits {\rm (AAT)}.
\end{proposition}

\begin{proof}
Letting $(z, w) \in {\mathbb C}^{\pi } \times {\mathbb C}^{\pi }$,
we define
$$
\widetilde{K} := {\mathbb C}(\xi _1(z), \dots , \xi _{\pi }(z), \eta_1 (z), \dots , \eta_{\pi } (z),
\xi _1(w), \dots , \xi _{\pi }(w), \eta _1(w), \dots , \eta _{\pi }(w)).
$$
We have a subfield $\widetilde{\kappa }$ of ${\rm Mer}(A \times A)$ such that
$\widetilde{K} = (\sigma \times \sigma )^{*}\widetilde{\kappa }$.
Take any $\psi \in \kappa $. It suffices to show that
$\psi (\varphi (P) + \varphi (Q)) \in (\varphi \times \varphi )^{*}\widetilde{\kappa }$.
\par
We take $P, Q \in (X \setminus S)^{(\pi )}$ such that $\varphi (P)$ and $\varphi (Q)$
are generic points of $A$.
Fixing $\varphi (Q)$, we consider $\psi (\varphi (P) + \varphi (Q))$ as a function
of $P$. Then it is meromorphically extended to $(X_{{\mathfrak m}})^{(\pi )}$.
Similarly, it is a meromorphic function on $(X_{{\mathfrak m}})^{(\pi )}$ as a
function of $Q$ if we fix $\varphi (P)$. The set of $P \in (X \setminus S)^{(\pi )}$
such that $\varphi (P)$ is not generic is an analytic set of positive codimension.
Furthermore, $X^{(\pi )} \setminus (X \setminus S)^{(\pi )}$
is an analytic subset of $X^{(\pi )}$. Then $\psi (\varphi (P) + \varphi (Q))$ extends meromorphically
to $X^{(\pi )} \times X^{(\pi )}$ (for example, see
Theorem 2 in \cite{ref10} or Corollary 3.4 in \cite{ref7}).
There exists $\widetilde{\xi } \in {\rm Mer}(X^{(\pi )} \times X^{(\pi )})$ such that
${\rm Mer}(X^{(\pi )} \times X^{(\pi )}) = {\rm Mer}((X_{{\mathfrak m}})^{(\pi )} \times
(X_{{\mathfrak m}})^{(\pi )})(\widetilde{\xi })$. Then $\psi (\varphi (P) + \varphi (Q))$
is a rational function of generators of ${\rm Mer}((X_{{\mathfrak m}})^{(\pi )} \times
(X_{{\mathfrak m}})^{(\pi )})$ and $\widetilde{\xi }$.
However, it is independent of $\widetilde{\xi }$ on
$\left((X_{{\mathfrak m}})^{(\pi )} \times (X \setminus S)^{(\pi )}\right)
\cup \left( (X \setminus S)^{(\pi )} \times (X_{{\mathfrak m}})^{(\pi )} \right)$.
Therefore we obtain
$\psi (\varphi (P) + \varphi (Q)) \in (\varphi \times \varphi )^{*}\widetilde{\kappa }$.
\end{proof}

Let $\overline{A}$ be the standard compactification of $A$. Then a W-type subfield
$\sigma ^{*}({\rm Mer}(\overline{A})|_{A})$ is considered as a degenerate
abelian function field, for we have the following theorem.

\begin{theorem}
We have $K \cong \sigma ^{*}({\rm Mer}(\overline{A})|_{A})$,
hence ${\rm Mer}( (X_{{\mathfrak m}})^{(\pi )}) \cong
{\rm Mer}(\overline{A})$.
\end{theorem}

\begin{proof}
Since $K$ admits {\rm (AAT)} (Proposition 5.1), there exists a
${\mathbb C}$-linear isomorphism $\Phi : {\mathbb C}^{\pi }
\longrightarrow {\mathbb C}^{\pi }$ such that 
$K_0 := \Phi ^{*}K$ is a W-type subfield by Theorem 2.5.
Let $\Gamma _{K_0}$ be the period group of $K_0$. Then we have
$$ B := {\mathbb C}^{\pi }/\Gamma _{K_0} = {\mathbb C}^{p'}
\times ({\mathbb C}^{*})^{q'} \times {\mathcal Q}',$$
where ${\mathcal Q}' = {\mathbb C}^{r'}/(\Gamma _{K_0})_0$ is an
$r'$-dimensional
quasi-abelian variety of kind 0 and $p' + q' + r' = \pi $.
Furthermore we have 
$K_0 = \sigma _{K_0}^{*}({\rm Mer}(\overline{B})|_{B})$ by
the definition of W-type subfields, where 
$\sigma _{K_0} : {\mathbb C}^{\pi } \longrightarrow B$ is the projection.
\par
From $K \subset \sigma ^{*}{\rm Mer}(A)$ and $A = {\mathbb C}^{\pi }/\Gamma $
it follows that $\Gamma \subset \Gamma _{K}$.
We note $\Phi (\Gamma _K) = \Gamma _{K_0}$. We set 
$\widetilde{\kappa } := (\varphi ^{-1})^{*}\left( {\rm Mer}(X^{(\pi )})|_
{(X \setminus S)^{(\pi )}}\right) \subset {\rm Mer}(A)$.
Since $\varphi ^{*}\kappa = {\rm Mer}((X_{{\mathfrak m}})^{(\pi )})|_
{(X \setminus S)^{(\pi )}} \subset {\rm Mer}(X^{(\pi )})|_{(X \setminus S)^{(\pi )}}$,
we have $\kappa \subset \widetilde{\kappa }$.
If we set $\widetilde{K} := \sigma ^{*}\widetilde{\kappa }$, then we have
$K \subset \widetilde{K} \subset \sigma ^{*}{\rm Mer}(A)$. Therefore,
$\Gamma \subset \Gamma _{\widetilde{K}} \subset \Gamma _K$,
where $\Gamma _{\widetilde{K}}$ is the period group of
$\widetilde{K}$. Hence we obtain the following sequence of homomorphisms:
$$\begin{CD}
A @>\tau _1>> {\mathbb C}^{\pi }/\widetilde{K} @>\tau _2>>
{\mathbb C}^{\pi }/\Gamma _{K}.\\
\end{CD}
$$
Since ${\rm Trans}_{{\mathbb C}}K = {\rm Trans}_{{\mathbb C}}\widetilde{K}
= \pi $, $\tau _2 : {\mathbb C}^{\pi }/\widetilde{K} \longrightarrow 
{\mathbb C}^{\pi }/\Gamma _{K}$ is an isogeny (Proposition 3 in \cite{ref5}).
By Proposition 7.1 in \cite{ref6} we have
$\varphi ^{*}({\rm Mer}(\overline{A})|_{A}) \subset {\rm Mer}(X^{(\pi )})|_
{(X \setminus S)^{(\pi )}}$. Then we obtain
${\rm Mer}(\overline{A})|_{A} = (\varphi ^{-1})^{*}(\varphi ^{*}({\rm Mer}(
\overline{A})|_{A}) \subset \widetilde{\kappa }$.
Therefore, we have $\sigma ^{*}({\rm Mer}(\overline{A})|_{A}) \subset
\widetilde{K}$. Since $\sigma ^{*}({\rm Mer}(\overline{A})|_{A}) $ is a W-type
subfield, we have 
$${\rm Trans}_{{\mathbb C}}\sigma ^{*}({\rm Mer}(\overline{A})|_{A}) 
= {\rm Trans}_{{\mathbb C}} \widetilde{K} = \pi .$$
The period group of $\sigma ^{*}({\rm Mer}(\overline{A})|_{A})$ is
$\Gamma $. Then $\tau _1 : A = {\mathbb C}^{\pi }/\Gamma \longrightarrow
{\mathbb C}^{\pi }/\Gamma _{\widetilde{K}}$ is also an isogeny by
Proposition 3 in \cite{ref5}. Therefore, 
$\tau _2 \circ \tau _1 : A \longrightarrow {\mathbb C}^{\pi }/\Gamma _K$
is an isogeny. We note that $B$ and ${\mathbb C}^{\pi }/\Gamma _K$ are
isogenous for $K_0 = \Phi ^{*}K$. Thus we see that $A$ and $B$ are isogenous.
Since an isogeny is fiber preserving (Proposition 2 in \cite{ref5}), we have
${\rm Mer}(\overline{A})|_{A} \cong {\rm Mer}(\overline{B})|_{B}$ and 
${\rm Mer}(\overline{B})|_{B} \cong {\rm Mer}\left( \overline{
({\mathbb C}^{\pi }/\Gamma _K)}\right)\bigr|_{{\mathbb C}^{\pi }/\Gamma _K}$.
Hence we obtain $K \cong \sigma ^{*}{\rm Mer}(\overline{A})|_{A}$.
\end{proof}

The generators $\{ \xi _i (z), \eta _i(z) ; i = 1, \dots , \pi \}$ is just the
fundamental degenerate abelian functions belonging to
${\rm Mer}(X_{{\mathfrak m}})$ as in the non-singular case.



\begin{thebibliography}{9}
\bibitem{ref1}
Y. Abe,
Meromorphic functions
admitting an algebraic addition theorem,
Osaka J. Math.,
\textbf{36}
(1999),      
343--363.      

\bibitem{ref2}
Y. Abe,
A statement of Weierstrass on meromorphic functions
which admit an algebraic addition theorem,
J. Math. Soc. Japan,
\textbf{57}
(2005),      
709--723.      

\bibitem {ref3} Y. Abe, 
Explicit representation of degenerate abelian
functions and related topics,
Far East J. Math. Sci.,
\textbf{70}
(2012),
321--336.

\bibitem {ref4} Y. Abe,
Toroidal Groups, 
Yokohama Publishers, Inc., Yokohama, 2018.

\bibitem {ref5} Y. Abe,
Meromorphic function fields closed by partial derivatives,
to appear in Acta Sci. Math. (Szeged).

\bibitem {ref6} Y. Abe,
Analytic study of singular curves,
preprint, arXiv: 1609.04517.

\bibitem {ref7} M. Jarnicki and P. Pflug, 
An extension theorem for separately meromorphic functions
with pluripolar singularities,
Kyushu J. Math.,
\textbf{57}
(2003),
291--302.

\bibitem {ref8} P. Painlev\'e,
Sur les fonctions qui admettent an th\'eor\`eme
d'addition,
Acta Math.,
\textbf{27}
(1903),
1--54.

\bibitem {ref9} J.-P. Serre, 
Groupes alg\'ebriques et corps de classes,
Hermann, Paris, 1959.

\bibitem{ref10}
B. Shiffmann,
Separately analyticity and Hartogs theorems,
Indiana Univ. Math. J.,
\textbf{38} 
(1989),      
943--957.      


\end{thebibliography}
\end{document}